\documentclass[11pt]{article}
\usepackage{graphicx, amsmath, amssymb, amsthm}
\usepackage[margin=2cm]{geometry}
\usepackage{commath}       
\usepackage{amsfonts}

\usepackage[margin=2cm]{geometry}
\usepackage{setspace}

\usepackage{algorithmic}
\usepackage{algorithm}
\usepackage{placeins}
\usepackage{authblk}       			  % title
\usepackage{subfigure}

\usepackage{hyperref}                 % Note - it requires to COMPILE TWICE !!

\newcommand{\spd}{{\cal{SPD}} (n)}
\newcommand{\spdn}[1]{{\cal{SPD}} {(#1)}} 
\newcommand{\so}{{\cal{SO}} (n)}
\newcommand{\son}[1]{{\cal{SO}} {(#1)}} 
\newcommand{\gl}[1]{{\cal{GL}} {#1}} 
\newcommand{\ltri}{{\cal{L}} (n)}
\newcommand{\utri}{{\cal{U}} (n)}
\newcommand{\dMat}{{\cal{D}} (n)}
 
\newcommand{\gmat}{{\mathsf{M}}}

\DeclareMathOperator{\tr}{tr}
\DeclareMathOperator{\diag}{diag}

\def\Acknowledgement{\goodbreak\bigskip\noindent{\bf Acknowledgement.\ }}

\newtheorem{theorem}{Theorem}[section]
\newtheorem{lemma}[theorem]{Lemma}
\newtheorem{corollary}[theorem]{Corollary}
\newtheorem{proposition}[theorem]{Proposition}
\newtheorem{definition}{Definition}[section]
\newtheorem{remark}{Remark}[section]
\newtheorem{example}{Example}[section]

\begin{document}
%---------------------------------------------------------------

% %------- For editting--------
% \setcounter{page}{0}
% \newpage
% \tableofcontents
% \newpage
% %-------------------------

%***************** Title *****************************************
\title{Approximation operators for matrix-valued functions based on matrix decompositions}

\author[1]{Nira Dyn\thanks{niradyn@post.tau.ac.il} }
\author[ ]{Uri Itai\thanks{uri.itai@gmail.com}}
\author[2]{Nir Sharon\thanks{nsharon@math.princeton.edu} }
\affil[1]{\footnotesize{School of Mathematical sciences, Tel-Aviv University, Tel-Aviv, Israel}}
\affil[2]{\footnotesize{The Program in Applied and Computational Mathematics, Princeton University, Princeton NJ, USA}}

\date{}
\maketitle
%***************************************************************

%***************** Abstract ************************************
\begin{abstract}
Given a set of matrices, modeled as samples of a matrix-valued function, we suggest a method to approximate the underline function using a product approximation operator. This operator extends known approximation methods by exploiting the structure of the matrices in the samples set, and based on decomposition theorems. We introduce our approach in detail and discuss its advantages using a few examples. In addition, we provide basic tools for analyzing properties of the matrix functions generated by our approximation operators.
\end{abstract}

%================================ Introduction ===============================

\section{Introduction}

Matrices are ubiquitous in various fields of mathematics, engineering and science, and so are the variety of related mathematical concepts. In particular, many applications from classical geophysics \cite{elliott1970determination} through computer graphics \cite{chefd2002constrained} to modern medical imaging \cite{arsigny2007geometric}, use matrix-valued functions (MVFs) that have to be approximated from a finite set of samples.

A matrix-valued function is a mapping $ f \colon \mathbb{R}^d \to \gmat $, where $\gmat$ is a class of matrices. In this paper we address the problem of approximating univariate MVFs ($d=1$). Specifically, we interest in approximations of MVFs which retain some of the special properties of the matrix class $\gmat$. 

Interpolation of MVFs is not a new concept. In \cite{beals1984scattering,delsarte1979nevanlinna}, several smooth solutions are presented and analyzed. The analysis includes many important results such as conditions for existence and uniqueness, but with only a minor focus on the algebraic and geometric features of the interpolant. Additional papers formulate the analogous of classical approximation tools, such as Taylor approximation for MVFs \cite{mathias1993approximation}, and the study of the polynomial interpolation and Guassian curvature for MVFs \cite{sinap1994polynomial}. Approximation of MVFs is also popular in other fields of applied research. For instance, in the design of electronic circuits, \cite{feldmann1995reduced} study an adaptation of Pad\'e approximation for matrices to compute models of linear circuits. 

For matrix classes equipped with the structure of a Riemannian manifold, many more approximation methods are available as the approximation of manifold-valued data have become popular in recent years, e.g., \cite{shingel2009interpolation,WallnerNiraProx,ThomasLie2}. These papers use different techniques for adapting advanced approximation operators, such as subdivision schemes, from numbers to elements on a manifold. Inspired by the variety of solutions, we wish to use the available operators, typically defined on special structured matrix classes or matrix manifolds, and to extend them to additional matrix classes using matrix decompositions.

Applying matrix decomposition to problems in approximation of MVFs is not a new concept. An important example for such an application is the construction of a continuous extension for the SVD decomposition presented in \cite{bunse1991numerical}. This paper opened the door for many computational techniques for calculating this continuous extension and its use for approximation too, e.g. see \cite{ay2010analytic} and reference therein. Additional examples are in \cite{dieci1999smooth} where they use several types of matrix decompositions to construct smooth continuous extensions, and in \cite{chen2003analysis} where thay apply the spectral decomposition for the analysis of a certain class of MVFs. In computer graphics, \cite{Alexa2000} introduces the notion of ``as-rigid-as-possible" for transforming shapes. One interpretation of this approach is the ability of factorizing the matrix that represents a transformation by the polar decomposition for applying a linear interpolation which describes the transition between this matrix and the unit matrix. By doing so, they utilize matrix decomposition to control certain properties related to the matrices in use. This approach best resembles the spirit of our construction.

\subsection{Problem formulation and our approach} 

The problem discussed in this paper is as follows. Let $\mathbf{A} = \{ A^{(i)} \}_{i \in \mathcal{J} }$, $\mathcal{J} \subset \mathbb{Z}$, be a sequence of matrices from the matrix class $\gmat$, sampled from a MVF, $f \colon \mathbb{R} \to \gmat$, on a compact segment $\mathtt{I} \subset \mathbb{R} $. Namely, 
\[  f(t_i) = A^{(i)} , \quad i \in \mathcal{J},  \quad t_i \in \mathtt{I}  , \] 
 with $t_i<t_{i+1}$, $i \in \mathcal{J}$. Fit a continuous MVF, $\Gamma \colon \mathtt{I} \to \gmat$ such that 
\begin{enumerate}
\item (Approximation)
\[\Gamma(t) \approx f(t), \quad t \in \mathtt{I}.\]
\item (Preservation of properties)
The matrices $\Gamma(t)$, $t \in \mathtt{I}$ have a set of predetermined properties of the matrices $A^{(i)}$, $i \in \mathcal{J}$. For example, positive determinant
$\det(A^{(i)})>0$, $i \in \mathcal{J} $ leads to  $\det(\Gamma(t))>0$, $t \in \mathtt{I}$. 
\end{enumerate}
We address this problem from a high-level point of view, as we construct our approximation operators based on other approximation operators that are available on other, potentially simpler domains. This is done using matrix decompositions.

Matrix decompositions are powerful tools in any practical area that involves matrices. For a given problem, we choose a specific decomposition to satisfy two main requirements. First is to preserve the invariant found in the sampled MVF. Second, the decomposition needs to reduce the approximation problem by solving it for each component separately. The assumption is that on these domains we are equipped with well-established approximation operators. The resulting operator is termed \textit{product approximation operator}. We study its construction, demonstrate it on several cases, and develop tools for analysis. In our approach the quality of $\Gamma(t)$ should not measured solely by the distance to $f(t)$ (the approximation error) but also by the invariant properties of $f(t)$ that are being preserved in $\Gamma(t)$. 

In addition, we show one variant of our construction, where we use matrix decomposition to reduce the approximation problem to easier domain, not in a product fashion but rather in a ''square root`` manner. By doing it we manage to relax a challenging task as approximating MVFs over triangular matrices by mapping it to a much easier space.

\subsection{Outline} 

The paper is organized as following. Section~\ref{sec:Theoretical_Background} provides the notation and theoretical background, including a short survey on some of the available approximation operators for MVFs. In Section~\ref{sec:app_Schemes} we introduce our construction of product approximation operators, to be discussed throw examples in Section~\ref{sec:TwoExamples} where we describe the specific construction of three such operators. In Section~\ref{sec:noteOnAnlaysis} we provide some fundamental tools for analysis of our product operators. We conclude the paper in Section~\ref{sec:triangular_matrices} with an interesting variant of our method for triangular matrices.

%========================= Theoretical Background ===============================

\section{Theoretical background} \label{sec:Theoretical_Background}

%---------------------------------------------------------
\subsection{Notation}

We begin with a few elementary notations and definitions related to matrices that we give here. We denote by 
$A^\ast$ the standard transpose, $A^\ast_{i,j} = (A_{j,i})^\ast$. The standard trace operator is given by $\tr(A) = \sum_{i=1}^n A_{i,i}$, and $\diag(A)=(A_{1,1},\ldots,A_{n,n})$ is the vector of the diagonal elements of $A$. For a square $n \times n$ matrix $A$, the determinant is $\det(A)$. In Table \ref{tab:NotationMatrixSpaces} we include the notation we use for several important classes of matrices. Unless otherwise stated, we consider matrices with real elements. A general matrix class is denoted by $\gmat$, usually consists of square matrices. If we want to emphasize the order we use $\gmat(n)$ or $\gmat(n,m)$ for non-square matrices of order $n \times m$.

Throughout the paper, we use the bold notation $\mathbf{A} $ for the sequence of data $\{ A^{(i)} \}_{i \in \mathcal{J}}$ from the matrix class $\gmat$. Similarly, when two or more matrix classes are discussed, we use the notation $\mathbf{A}_j = \{ A_j^{(i)} \}_{i \in \mathcal{J}}$ for the data in the class $\gmat_j$.

%\begin{center}
\begin{table}[ht] \centering
  \begin{tabular}{c|c }
    The notation &  The class of \\\hline 
      $\spd$ & symmetry positive definite matrices \\ 
      $\so$  & special orthogonal matrices with determinant $1$ \\
      $\gl(n)$   & invertible matrices \\
     $\ltri$ ($\utri$) & invertible, lower (upper) triangular matrices \\
%     $\utri$  & upper, invertible triangular matrices \\
     $\dMat$    & diagonal matrices
  \end{tabular}
  \caption{Notation for special classes of square matrices of order $n$.}
  \label{tab:NotationMatrixSpaces}
\end{table}
%\end{center}

%---------------------------------------------------------
\subsection{Approximation of MVFs} \label{subsec:approx_MVFS}

The classical computation methods for the approximation of a univariate function from its samples, such as polynomial or spline interpolation, are linear. These methods cannot cope with most of MVFs defined over non-linear classes of matrices, e.g., the class (cone) of positive define matrices. Thus, the available methods for such matrix-valued data are based on different adaptations of the linear methods.

There are several different methods for the adaptation of a sampled based linear approximation operator to matrix-valued samples. Here we present three ``popular" methods, all ``intrinsic" to the underline matrix class, which means that the generated matrices in the approximation are guaranteed to belong to the same original matrix class.

First is the log-exp mappings, defined for the classes of matrix Lie groups, e.g., \cite{WallnerNiraProx}. This method consists of three steps: projecting the samples into the corresponding Lie algebra, applying the linear operator to the projected samples in the Lie algebra, and projecting the approximant back to the Lie group. There are several computational difficulties in the realization of this ``straightforward" idea, mainly in the evaluation of the logarithm and exponential of a matrix, see e.g., \cite{shingel2009interpolation}. We briefly discuss these issues in Section \ref{sec:triangular_matrices}. 

A similar idea applies for local approximations, such as the output of subdivision schemes, where the approximant at a given point depends only on samples in the neighborhood of the point. In such a setting, the exp-log method applies, with the Lie algebra replaced by the tangent space at a point on the manifold, where an inherent difficulty in this approach is the choice of the location of the tangent space, e.g., \cite{ThomasLie2}. %For example, in the case of positive define matrices, the tangent space is the linear space of symmetrical matrices, where classical linear methods are well-defined.

The third method is by using repeated binary geodesic averages. To be specific, a linear sampled based approximation operator of the form $\mathcal{A}(f)(t) = \sum_{j=1}^n a_j(t)f(t_j)$ with $\sum_{j=1}^n a_j(t) = 1$,
can be rewritten in terms of repeated weighted binary averages in several ways \cite{WallnerNiraProx}. Then, equipped with one of these representations of $A(f)$ and an intrinsic average (average that is closed on the matrix class), we can replace each average between numbers with that intrinsic average of matrices. An example of such a representation is the de Casteljau algorithm for the evaluation of the approximating Bernstein polynomials, e.g., \cite{NirUri}.

%========================= Subdivision schemes based on matrix decomposition =============
\section{Approximation operators based on matrix decomposition} \label{sec:app_Schemes}

Let $\gmat$, $\gmat_1$, and $\gmat_2$ be matrix classes such that for any matrix $A \in \gmat$ there exists a unique decomposition
\begin{equation} \label{eqn:matrixDecomposition} 
A = A_1 A_2 ,\quad A_1 \in \gmat_1 , \quad A_2 \in \gmat_2  .  
\end{equation}
$\gmat$ is the \textit{product space} of $\gmat_1$ and $\gmat_2$. There is a handful of such matrix decomposition available, for example consider the QR-decomposition \cite[Chapter 5]{GolubMatrixBook}, where $\gmat=\gl(n)$, $\gmat_1 = \so $, and $\gmat_2$ is the class of $n$-th order upper triangular matrices with positive diagonal elements (in the context of Lie groups this decomposition is known as Iwasawa decomposition). 

Another example is the LU decomposition \cite[Chapter 3]{GolubMatrixBook} over $\gmat= \gl(n)$, where $\gmat_1 = \{ L \in \ltri \mid \diag(L)=(1,\ldots,1) \} $ and $\gmat_2 = \utri$. Two variants of this decomposition will be discuss in Subsection~\ref{sec:Preserving_signs} and in Section~\ref{sec:triangular_matrices}.

\begin{remark}
The matrix relation \eqref{eqn:matrixDecomposition} and product space $M$ can be regarded as a special case of Cartesian product and Cartesian product space. Throughout this paper, we focus on the matrix properties and operators on MVFs. This focus is different from the perspective of classical, topological studies of Cartesian products. Thus, we did not use the classical terminology and notations.
\end{remark}

\begin{remark}
A general decomposition is of the form
\[ A = \Pi_{j=1}^m  A_j, \quad A_j \in \gmat_j . \]
Here, for brevity we only use $m=2$. However, all the general results we obtain can be easily extended for $m>2$. In fact, we do discuss a few important and interesting special cases of $m=3$ in Section~\ref{sec:TwoExamples}.
\end{remark}

In the context of approximation, we require our matrix class to be associated with a metric. Some global metrics such as the Euclidean metric induced by the Frobenius norm, 
\begin{equation} \label{eqn:EuclideanMetric}
d(A,B) = \| A-B \|_F , \quad A,B \in \mathbb{R}^{n \times m} ,
\end{equation}
are always available, where $\norm{X}_F = \sqrt{\tr(XX^\ast)}$. However, they usually do not reflect the special geometry of each matrix space in use. For example, in matrix spaces that have the structure of Riemannian manifolds, such as $\so$ or $\spd$, one can use a Riemannian metric.

As assumed, $\gmat_1$ and $\gmat_2$ are special structured matrix classes or matrix manifolds, and denote by $d_1$ and $d_2$ the metrics associated with $\gmat_1$ and $\gmat_2$, respectively. Then, we define a product metric associated with the product space $\gmat$, based on $d_1$ and $d_2$, and their induced geodesics. These geodesics are defined on metric spaces via the concept of \textit{metric property}.

The metric property in metric spaces characterizes a geodesic, denote by $\gamma(t)$. This path is termed geodesic although it is not defined as the solution of Euler-Lagrange equations (there is not enough structure to do so) since it locally minimizes the distance between two points in the space $A,B$ with respect to a given metric $d$. The metric property is
\begin{equation} \label{eqn:metric_property}
 d(\gamma(t),B) = (1-t)d(A,B),\quad  0\le t \le 1 .  
\end{equation}
Clearly, we have $\gamma(0)=A$ and $\gamma(1)=B$. Furthermore, in neighborhoods (or spaces) where the geodesic is unique we get by symmetry and the triangle inequality that the above is equivalent to $ d(A,\gamma(t)) = td(A,B)$. In the following discussion we assume uniqueness of this geodesic (or a canonical way to choose one).

We suggest a method to construct a homogeneous product metric for $\gmat$ based on an auxiliary function and the metrics $d_1$, $d_2$. Let $A_{j}$ and $B_{j}$, $ j=1,2$ be two pairs of matrices, connected by two geodesics with respect to $d_j$ that satisfy \eqref{eqn:metric_property},
\[ \gamma_j \colon  [0,1] \to \gmat_j , \quad \gamma_j(0) = A_{j}, \quad \gamma_j(1) = B_{j} , \quad j=1,2 . \]
Then, we define a product geodesic by
\[ \gamma \colon [0,1] \to \gmat, \quad \gamma = \gamma_1 \gamma_2  , \quad \gamma(0) = A =A_1A_2 , \quad \gamma(1) = B=B_{1}B_2 . \] 
The following results shows how to construct a metric for $\gmat$ that satisfies the metric property \eqref{eqn:metric_property}.

\begin{proposition} \label{pro:HomoProductMetric}
Let $(\gmat_1,d_1)$,$(\gmat_2,d_2)$ be two metric spaces of matrices, and define a product matrix space by \eqref{eqn:matrixDecomposition}. Assume $ \psi \colon \mathbb{R}_+ \times \mathbb{R}_+ \to \mathbb{R}_+$ is a continuous function satisfying 
\begin{enumerate}
\item $\psi \ge 0$, and $\psi(x_1,x_2)=0$ if and only if $x_1=x_2=0$. 
\item $\psi$ is monotone in every variable.
\item $\psi$ is Jensen-convex, that is 
 \[ \psi \bigl(  (\frac{1}{2} x_1,\frac{1}{2} x_2)+ (\frac{1}{2} y_1, \frac{1}{2}y_2)  \bigr) \le \frac{1}{2} \bigl( \psi(x_1,x_2)+\psi(y_1,y_2)\bigr) . \]
\item $\psi$ is homogeneous i.e., $\psi(\alpha x_1,\alpha x_2)=\alpha \psi(x_1,x_2)$. 
\end{enumerate}
 Then,
 \begin{equation} \label{eqn:d_psi_metric}
  d_\psi(A,B)=\psi(d_1(A_1,B_1),d_2(A_2,B_2)) 
 \end{equation}
 is a product metric, and the product geodesic $\gamma = \gamma_1 \gamma_2$ in $\gmat$ has the metric property related to $d_\psi$.
\end{proposition}
\begin{proof}
We first need to show that $d_\psi$ is indeed a metric. The positivity is ensured by the first property of $\psi$, combined with the fact that $d_1$ and $d_2$ are metrics. The letter also implies the symmetry. For the triangle inequality, note that the Jensen-convexity and the homogeneity suggest
\begin{equation} \label{eqn:d_psi_proof}
\psi\left( (x_1,x_2)+(y_1,y_2)\right) \le \psi\left(x_1,x_2\right)+\psi\left(y_1,y_2\right) .
\end{equation}
Thus, for any $A,B,C \in \gmat$ we have
\begin{eqnarray*}
d_\psi(A,C) &=& \psi(d_1(A_1,C_1),d_2(A_2,C_2)) \\
& \le &   \psi(d_1(A_1,B_1)+d_1(B_1,C_1),d_2(A_2,B_2)+d_2(B_2,C_2)) \\
& \le &   \psi(d_1(A_1,B_1),d_2(A_2,B_2)+\psi(d_1(B_1,C_1),d_2(B_2,C_2))  = d_\psi(A,B) + d_\psi(B,C) ,
\end{eqnarray*}
where for the first inequality we use the monotonicity of $\psi$ and for the second we use \eqref{eqn:d_psi_proof} together with the positivity of $\psi$. The notations are as in \eqref{eqn:matrixDecomposition} with $C=C_1C_2$. 

The metric property of the product geodesic is derived by the homogeneity of $\psi$. Indeed, with the notation of the above discussion, for the product geodesic $\gamma$,
\[
d_\psi(A,\gamma(t))=\psi(d_1(A_1,\gamma_1(t)),d_2(A_2,\gamma_2(t))) =\psi(td_1(A_1,B_1),t d_2(A_2,B_2) )=td_\psi(A,B) . \]
\end{proof}

Note that by the Euler equation for homogeneous functions of order $n$, $\frac{\partial \psi}{\partial x_1} x_1 + \frac{\partial \psi}{\partial x_2} x_2 = n \psi  $ and so
if $\psi$ is twice differentiable then the kernel of the Hessian of $\psi$ contains the vector $(x_1,x_2)$ (our case is $n=1$). Therefore, we can only have a weak inequality in the  Jensen-convex condition of Proposition~\ref{pro:HomoProductMetric}. 

A classical example for product metric is the $p$-product metric 
\begin{equation} \label{eqn:p_product_metric}
d_p(A,B) = (d_1(A_1,B_1)^p + d_2(A_2,B_2)^p)^{1/p} ,\quad  A= A_1A_2, \quad B = B_1B_2 ,
\end{equation}
with  $1\le p \le \infty$. These examples are a special case of Proposition \ref{pro:HomoProductMetric}, with $\psi(x,y) = \bigl( x^p+y^p \bigr)^{\frac{1}{p}} $, where the conditions on $\psi$ follow from the Minkowski inequality.

The product metric will have an important role in our analysis of product operators, as we will see next in Section~\ref{sec:noteOnAnlaysis}.

In many cases, $\gmat_1$ and $\gmat_2$ are well-structured matrix classes (e.g, Lie groups) having several known approximation operators available. To define the product approximation operator denote BY $\mathcal{S}_1$ and $\mathcal{S}_2$ two approximation operators on $\gmat_1$ and $\gmat_2$, respectively. Furthermore, denote by $\Gamma_j = \mathcal{S}_j(\mathbf{A}_j)$, $j=1,2$, their generated MVFs. Now, a product approximation operator $\mathcal{S}$ on $\gmat$ is defined as
\begin{equation} \label{eqn:productSchemes}
 \mathcal{S}(\mathbf{A})  = \mathcal{S}_1(\mathbf{A}_1) \mathcal{S}_2(\mathbf{A}_2) ,
\end{equation}
where the sequence $ \{ A^{(i)} \}_{i \in \mathcal{J} }$ is defined as $  A^{(i)} = A_1^{(i)}A_2^{(i)} $ for $ \mathbf{A}_j = \{ A_j^{(i)} \}_{i \in \mathcal{J} }$, $j=1,2$. It is worth mentioning that we do not assume that the operators are different (in cases where $\gmat_1=\gmat_2$).

%========================= Two examples for product schemes ========================

\section{Three examples of constructing product approximation operator} \label{sec:TwoExamples}

The vast variety of matrix decompositions enables to design many types of product approximation operators for many scenarios. We demonstrate how to construct such operators according to given invariant or matrix properties to be preserved.

\subsection{Preserving positive determinant}

We start with a simple example to illustrate the general setting. Consider we have data over $\gmat = \gl(n)$, consists of matrices of positive determinant. Since $\gmat$ is a Lie group, one can apply directly an approximation operator such as geodesic subdivision schemes \cite{dyn2017manifold}. We describe an alternative which guarantee the preservation of positive determinants and also facilitate bounding the specific behavior of the determinant in the generated MVF.

One generalization of the polar representation of a complex number $z = re^{i\theta}$ is the polar decomposition of a matrix $A$,
\begin{equation} \label{eqn:PolarDecomposition}
A=P_A Q_A , \quad P_A \in \spd , \quad Q_A \in \so . 
\end{equation}
For invertible $A$ we have $P_A = (AA^\ast )^{\frac{1}{2}}$, and $Q_A = P_A^{-1}A$ (orthogonality of $Q_A$ is guaranteed as well as $\det(Q_A)=1$). This unique decomposition implies that a product operator is well-defined with one operator on $\spd$ and a second operator on $\so$. Both classes are well-studied manifolds, having many approximation operators available, see e.g., \cite{NirUri} and \cite{shingel2009interpolation} on $\spd$ and on $\so$, respectively.

Define the approximation separately on $\spd$ and on $\so$ ensures positive determinants in the generated MVF. Furthermore, bounds on the determinants are equivalent to bound the determinants of the generated MVF over $\spd$. However, over $\spd$ this kind of bounds are easier to calculate and in some cases even available directly, e.g., in \cite{UriNir}.

\subsection{Controlling a geometrical artifact in approximation of ellipsoids} \label{subsec:artefact}

The set of symmetric positive definite (SPD) matrices has several interesting geometrical interpretations. One is derived by associating each SPD matrix to a centered ellipsoid (the orthogonal eigenvectors as its axes and the corresponding eigenvalues as their lengths). This map is injective which means that The set of SPD matrices coincides with the set of centered ellipsoids. Other interpretation comes from the structure of a Riemannian manifold, which includes all of its ``natural" inherent geometry, such as the geodesic distance and geodesics. In particular, the geodesics on a Riemannian manifold induce an intrinsic average, where the mid-point of the geodesic connecting two matrices is defined as their average, see e.g. \cite{global_nira}. The geodesics in $\spd$, with respect to the Riemannian metric, have a closed form \cite{ALMean}. Therefore, they are used as an average in the construction of approximation operators, e.g. \cite{UriNir}. These approximation operators based upon geodesics have many important (algebraic, spectral, etc.) properties. However, as will see next, using this natural geometry of $\spd$ might lead to unnatural, geometric artifact in the first interpretation of $\spd$. 

\begin{figure}[t]   %[H]
\begin{center}
\subfigure[A transition of the matrices by rotation]{
\includegraphics[width=.7\textwidth]{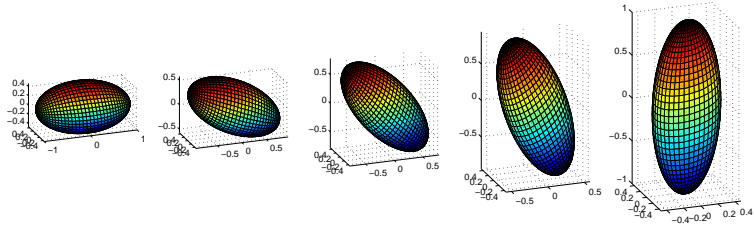}
\label{fig:tran1}
}
\subfigure[A transition on the geodesic connecting the most left anf right matrices, with respect to the Riemannian metric in $\spdn{3}$]{
\includegraphics[width=.7\textwidth]{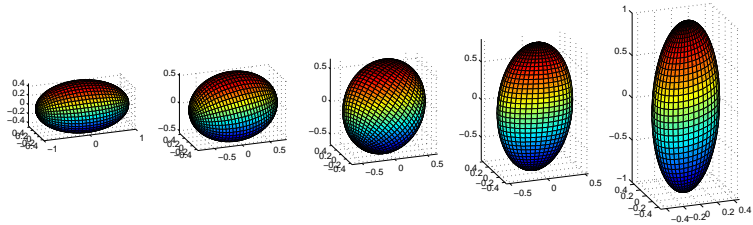}
\label{fig:tran2}
}
\caption{Two possible transitions between a positive definite matrix and its $90$ degrees rotation, appear as the first and last ellipsoids of every sequence.}
\label{fig:TransitionMatrices}
\end{center}
\end{figure}

Figure~\ref{fig:TransitionMatrices} demonstrates the geometric artifact of averaging positive definite matrices using their Riemannian geodesic, where two possible transition matrices between a matrix and its $90$ degrees rotation are presented. The matrices are shown by their interpretation as ellipsoids. In the first one, depicted in Figure~\ref{fig:tran1}, only rotations are ``allowed" and we naturally connect the two matrices (ellipsoids) by rotating them. In Figure~\ref{fig:tran2} we use the matrices on a geodesic of $\spd$ that connects the matrix and its $90$ degrees rotation, with respect to the Riemannian metric. By doing so, we first ``ball" the ellipsoid, and then stretch it to reach its rotated version. 

To overcome such a phenomenon, we have to impose the use of rotations, also known as ``rigid-as-possible", see e.g., \cite{sorkine2007rigid}. This approach is related to the spectral decomposition. The spectral decomposition of a matrix $A \in \spd$ can be written as
\begin{equation} \label{eqn:SpectralDecomposition}
A = Q^\ast D Q , \quad Q \in \so, \quad D \in \dMat , 
\end{equation}
where the diagonal elements of $D$ are in non-decreasing order. 

Golumb and Van Loan \cite[Chapter 12.4.1]{GolubMatrixBook} discuss the following problem of subspace rotation
\[  \min_{Q \in \so } \| A- BQ \|_F , \quad A,B \in \mathbb{R}^{m \times n} , \]
using the Frobenius norm \eqref{eqn:EuclideanMetric}. This problem is known as the ``orthogonal Procrustes problem" and is solved via the SVD decomposition of $B^\ast A$. However, for data consisting of matrices in $\spd$, we are interested in a variant of this problem which preserves symmetry. Schonemann called it the ``two-sided orthogonal Procrustes problem" \cite{schonemann1968two}. This problem is to find the minimizer $Q \in \so$ of
\begin{equation} \label{eqn:ProcrustesProblem}
 d_S(A,B) = \min_{Q \in \so } \| A- Q^\ast BQ \|_F , \quad A,B \in \spd . 
\end{equation}
The solution of \eqref{eqn:ProcrustesProblem} is $Q = P_1^\ast P_2$ such that $P_1,P_2 \in SO(n)$ and $P_1$ and $P_2$ diagonalize $A$ and $B$, respectively, with non-decreasing eigenvalues on the diagonal. Note that $d_S(A,B)$ is a pseudo-metric which induces the partition of $\spd$ into equivalence classes where $X \sim Y$ if and only if $d_S(X,Y)=0$, that is $X$ and $Y$ are orthogonal similar.

We define the product operator based on the spectral decomposition \eqref{eqn:SpectralDecomposition} as
\begin{equation} \label{eqn:SpectralSchemes}
\mathcal{S} = (\mathcal{S}_1)^\ast \mathcal{S}_2 \mathcal{S}_1 ,
\end{equation}
where $\mathcal{S}_1$ is an operator on $\so$, $\mathcal{S}_2$ is an operator on $\dMat$, and $(\mathcal{S}_1)^\ast$ is the transpose of the output of $\mathcal{S}_1$. From the ellipsoid point of view, the operator $\mathcal{S}_1 $ rotates the axes while the operator $\mathcal{S}_2$ scales them. Therefore, one can verifies that the product operator \eqref{eqn:SpectralSchemes} can be used to solve the geometric artifact problem described above by interpolating the rotations. To see it, consider a consistent $\mathcal{S}_2$, that is an operator that returns a constant function for constant input samples. Then, if all the data matrices are similar i.e., share the same eigenvalues, then so are all the matrices of the function generated by $\mathcal{S}$.  One meaning of the last observation is that if we use a local operator (say a subdivision scheme) and apply the product operator to a neighborhood of data which consists of similar matrices, the product operator will rotate the ellipsoids in this neighborhood while keeping the rigid structure of them. Moreover, if also $\mathcal{S}_1)$ is consistent, we have that any operator of the form \eqref{eqn:SpectralSchemes} is closed on data from the same equivalence class of $d_S$.

%---------------- Wallner's remark 2------------------------------
\begin{remark}
When the matrices of the data approach scalar matrix, that is a matrix of the form $\alpha I $, the output of the operator for $\so$ may not be continuous. The reason is that the orthogonal part of the spectral decomposition is not unique for scalar matrices. However, this is not reflected in the output of the product operator, since the $\so$ operator on the transpose is canceling out the discontinuity. In the sense of the upcoming analysis, the product metric is not majorized by any dense metric of $\spd$ (more on Section~\ref{sec:noteOnAnlaysis}). 
%Hence, we cannot use the machinery of Section \ref{sec:app_Schemes}, where the analysis is done under the assumption of a unique decomposition.
\end{remark} 

%---------------- Gromov part ------------------------------
\begin{remark}
The geometry motivation presented above, illustrates two different types of geometry, defined by the two distances. The first is the distance derived from the geodesics in the Riemannian manifold of $\spd$, denoted by $d_R$, and the second is the distance by the pseudo-metric $d_S$ of \eqref{eqn:ProcrustesProblem}. Nevertheless, an intermediate geometry, obtained from a hybrid metric is possible. That is,
\[ d_H(A,B) = d_R(A,B) + \beta d_S(A,B) , \quad A,B \in \spd . \]
For $\beta \in (0,\infty)$, $d_H$ is a metric. This parameter describes the ``freedom" to use rotations in the geodesic that connect the two matrices with respect to $d_H$.

Such operators, based on the geometry of the hybrid metric, can be calculated numerically and allow to control the above geometrical phenomenon in a continuous fashion between the two approaches of geodesic based operators ($d_R$ geometry) versus the product operators based on the spectral decomposition (the ``rigid'' $d_S$ geometry).
\end{remark} 

\subsection{Preserving signs of principle minors} \label{sec:Preserving_signs}

Assume the data $\mathbf{A}$ consists of a sequence of matrices with nonzero principle minors and we require to generate a continuous MVF that preserves this property in each of its parametric points. In addition, assume that the vector of signs of the principle minors is the same for all the data matrices (otherwise, due to Bolzano's Theorem, any continuous MVF that fits to the data will break the required property of nonzero principle minors. How can we address such a problem?

Nonzero principle minors of a matrix guarantee that we can apply Gaussian elimination process without pivoting. Thus, we exploit this structure by addressing the above problem with a product operator based on the matrix decomposition that related to the Gaussian elimination process -- this is the LDU decomposition, see e.g, \cite[Chapter 4]{GolubMatrixBook}. 

Let $A$ be an invertible matrix with nonzero principle minors. Then, the LDU decomposition (a variant of the LU factorization) defines three matrices $A=L_AD_AU_A$ where $L_A$ is a unit lower triangular matrix (with all diagonal elements equal to $1$), $U_A$ is a unit upper triangular matrix, and $D_A$ is a diagonal matrix with elements $(D_A)_{ii}= \frac{p_i(A)}{p_{i-1}(A)}$ $i=1,\ldots,n$, where $p_i(A)$ is the principle minor of $A$ of order $i$ and $p_0(A)=1$ \cite[Chapter 3]{GolubMatrixBook}.  

We use the product operator $\mathcal{S} = \mathcal{S}_1 \mathcal{S}_2 \mathcal{S}_3$, where $\mathcal{S}_1$ and $\mathcal{S}_3$ are operators on unit lower and upper triangular matrices, respectively. The class of unit triangular matrices (upper or lower) is called the Heisenberg group, and it is a Lie groups, where approximation operators are well-studied, see for example \cite{WallnerNiraProx}. Furthermore, we propose in Section~\ref{sec:triangular_matrices} an alternative for designing such operators in the spirit of product operators.

We choose the operator $\mathcal{S}_2$ on $\dMat$ as an element-wise operator (this is trivial since any operator based on matrix multiplication is such), that also preserves sign. Since in this problem we know that all scalar data consists of the same sign, this is equivalent to having a positivity preserving operator for scalars. Preserving positivity can be done for example by applying a logarithm to the (scalar) data, applying any approximation operator on numbers and mapping back using the exponential, see e.g., \cite{UriNir}. By retaining the signs of the diagonal elements and since $p_0(A)=1$ for all $A$, the signs of the principle minors of the MVF $\Gamma(t) = \mathcal{S}(\mathbf{A})(t)$ are the same as those of the data matrices, as required.

Note that the class of $\spd$ is a special case of data, corresponding to positive principle minors. 

%========================= Analysis tools =======================
\section{Notes on analysis of product approximation operators} \label{sec:noteOnAnlaysis}

\subsection{Regularity of product MVFs}

We start by providing basic definitions of the regularity for MVFs, given an intrinsic metric $d$.
\begin{definition}[Continuity] \label{def:Continuity}
Let $\Gamma \colon \mathtt{I} \to \gmat$ be a matrix-valued function satisfying,
\[  \lim_{ \triangle t \rightarrow 0} d(\Gamma(t),\Gamma(t+\triangle t)) =0, \quad t \in \mathtt{I}  . \]
Then $\Gamma$ is called a \textbf{\em continuous MVF} \normalfont{($\mathbf{C}^0$)} on $\mathtt{I}$. 
\end{definition}

\begin{definition}[$\alpha$-H\"older]  \label{def:Smoothness}
Let $\Gamma \colon \mathtt{I}  \to \gmat$ be a continuous MVF. $\Gamma$ is termed $\alpha$-H\"older continuous on $\mathtt{I}$ ($\mathbf{C}^{0,\alpha}$) if there exist constants $C>0$ and $0<\alpha \le 1$ such that
\[ d(\Gamma(t),\Gamma(t+\triangle t)) \le C |\triangle t|^\alpha, \quad  t,t+\triangle t \in \mathtt{I}  . \]
\end{definition}

\begin{remark}
It is possible to define a higher order of smoothness by the limit
\[ \frac{d}{dt} \Gamma (t) =  \lim_{\triangle t\rightarrow0}\frac{\Gamma(t+\triangle t)-\Gamma(t)}{\triangle t}, \quad t \in \mathtt{I} . \]
In the case where this limit exists and $\frac{d}{dt}\Gamma(t)$ is continuous for every $t \in \mathtt{I}$, we can define the MVF to be $\mathbf{C}^1$, and recursively define smoothness higher than $\mathbf{C}^1$. Nevertheless, two important issues should arise here. First, what is the matrix class to which the MVF $\frac{d}{dt}\Gamma(t)$ belongs to? Second, if $\gmat$ has a special structure such as the structure of a manifold embedded in $\mathbb{R}^{n \times n}$, then the continuity and the H\"older condition are measured intrinsically, while the derivative is measured extrinsically. So on this paper we do not consider higher smoothness.
\end{remark}

Let $\Gamma$, $\Gamma_1$, and $\Gamma_2$ be three MVFs, which in the view of \eqref{eqn:matrixDecomposition} satisfy
\begin{equation} \label{eqn:matrixCurveDecomposition}
\Gamma(t) = \Gamma_1(t) \Gamma_2(t) , \quad  \Gamma(t) \in \gmat , \quad \Gamma_1(t) \in \gmat_1 , \quad \Gamma_2(t) \in \gmat_2 , \quad t \in \mathbb{R} . 
\end{equation}
We study the regularity of $\Gamma(t)$ with respect to the given regularity of $\Gamma_1(t)$ and $\Gamma_2(t)$, where these regularities depend on the metrics in use. Therefore, we examine the relation between the metrics. 

Let $d$ and $\bar{d}$ be two metrics, defined on a common space $V$. We say that $\bar{d}$ is \textit{majorized} by $d$ if there exists a positive constant $C$ such that
\begin{equation} \label{eqn:majorized}
\bar{d}(u,v) \le C d(u,v) , \quad u,v\in V. 
\end{equation}
Two metrics are equivalent if they majorize each other. Using the above majorization we have,

\begin{lemma} \label{lemma:equiv_metrics}
Let $\Gamma:I \mapsto \gmat$ be a continuous ($\alpha$-H\"older) MVF, under the metric $d$, where $I \subset \mathbb{R}$ is a finite interval. If the product metric $d_\psi$ of Proposition~\ref{pro:HomoProductMetric} is majorized by $d$, then $\Gamma(t)$ is continuous ($\alpha$-H\"older) under $d_\psi$ as well.
\end{lemma} 
The proof is a straightforward use of the rules for the arithmetic of limits, Definition~\ref{def:Continuity}, Definition~\ref{def:Smoothness}, and the compactness of $I$. 

While matrix norms in finite dimensional spaces are equivalent, the metrics in general are not. To illustrate this claim we provide a counter example.
\begin{example}\label{Example:counterexample}
Let $\gmat = \gl(2)$. Consider the QR-decomposition and the MVF
\[  \gamma  \colon [0,1] \to \gl(2) , \]
defined by 
\[  \gamma(t) = Q(t)R(t) =
\begin{pmatrix} 
  \cos(t)     & \sin(t) \\ 
  -\sin(t) &  \cos(t)  
\end{pmatrix} 
\begin{pmatrix} 
  \exp(t)     & 0 \\ 
  0 &  \exp(t)  
\end{pmatrix} . \] 
For $\gmat$ we use the following metric, also known as the British Railway metric,
\[  d_{BR}(A,B) = \begin{cases} 0 &  A=B, \\
\|A\| + \|B\| &\text{otherwise.}
\end{cases} \]
On the one hand we have, 
\[  \lim_{t \rightarrow 0 }  Q(t) = \lim_{t \rightarrow 0} R(t) = I ,  \]
on the other hand,  there is no $\lim_{t \rightarrow 0 } \gamma(t)$ under the metric $d_{BR}$. Namely, $d_{BR}$ cannot be majorized by the standard Euclidean metric \eqref{eqn:EuclideanMetric} or any other equivalent metric of it. This example illustrates that the metric $d_{BR}$ is not equivalent to the metric \eqref{eqn:EuclideanMetric}.
\end{example}

We aim to inherit the regularity from the decomposition components to the product function. However, Example \ref{Example:counterexample} implies that extra conditions must be posed. In the following definition we introduce such a condition.
\begin{definition}
A metric space $(V,d)$ is called \textbf{metric dense} if for any $x \in V$ and any $\varepsilon>0$ there exist $y \in V$, $y \neq x$ such that $d(x,y)<\varepsilon$.
\end{definition}

The significance of the latter definition is illustrated in the next lemma.
\begin{lemma} \label{lemma:equivalent}
Let $(V,d_1)$ and $(V,d_2)$ be two metric dense matrix spaces. Then, in any compact neighborhood, the metrics $d_1$ and $d_2$ are equivalent.
\end{lemma}

\begin{proof}
Let $B_j(A,r) = \{ X \in V \mid d_j(A,X) \le r \} $, $j=0,1$ be the closed ball of radius $r$ around $A$, using the metric $d_j$. Since $V$ is of finite dimensions and the metrics are dense, these balls are well defined. Furthermore, we can find a positive constant $\rho $ such that $B_2(A,1) \subset  B_1(A,\rho)$. The opposite inclusion is analogous. Due to the compactness of the domain the claim follows. %See also \cite{}.
\end{proof}

The next conclusion is followed from the last discussion.
\begin{corollary}
Let $\Gamma_1$ and $\Gamma_2$ be continuous MVFs such that the metrics of each space is metric dense. Then, the product function $\Gamma(t)$ is a continuous MVF under any dense metric $d$.
\end{corollary}

\subsection{Properties of product operators directly inherited from their factory}

A fundamental question for iterative approximation operators (for example subdivision schemes) is the regularity of the generated approximation (see e.g., \cite{WallnerNiraProx}). Thus, as a conclusion from the previous subsection we have the following. 
\begin{corollary} \label{cor:general_regularity}
Suppose $\mathcal{S}_1$ and $\mathcal{S}_2$ generate continuous ($\alpha$-H\"older) MVFs, and assume the metric $d$ of $\gmat$ is majorized (see \eqref{eqn:majorized}) by some product metric $d_\psi$. Then, $\mathcal{S}$ generates continuous ($\alpha$-H\"older) MVFs. In such a case, we say that $\mathcal{S}$ is a $\mathbf{C}^{0}$ ($\mathbf{C}^{0,\alpha})$ product operator.
\end{corollary}

One method of measuring the quality of an approximation is the approximation order. We say that the approximation $\mathcal{S}$ has a $\nu$ order of approximation if the generated function $\Gamma(t) = \mathcal{S}(\mathbf{A})(t)$ satisfies 
\begin{equation} \label{eqn:app_order}
d(F(t),\Gamma(t)) \le C_F h^\nu ,
\end{equation}
where $C_F$ is a constant independent of $h$, but usually depends on $F$ and the matrix metric space $(\gmat,d)$. Here $h = ֿ\max_i \{ t_{i+1}-t_i \}$, where $t_i < t_{i+1}$ and $A_i = F(t_i)$, $i \in \mathcal{J}$.

A general approximation order for a product approximation operator is as follows.
\begin{theorem} \label{thm:apporder}
Assume $d$ of $\gmat$ is majorized by some product metric $d_\psi$, and consider a product operator of the form \eqref{eqn:productSchemes}, with $\mathcal{S}_1$ and $\mathcal{S}_2$ having approximation orders $\nu_1$ and $\nu_2$, with constants $C_1$ and $C_2$, respectively. Then, $\mathcal{S}$ has approximation order $\nu = \min \{ \nu_1, \nu_2 \}$ with a constant $C = \max \{ C_1,C_2 \} \psi(1,1)$.
\end{theorem}
\begin{proof}
By the homogeneity and monotonicity of $\psi$ we have 
\[ d_\psi (F(t),\Gamma(t)) = \psi(d_1(F_1(t),\Gamma_1(t)),d_2(F_2(t),\Gamma_2(t)))  \le  \psi (C_1 h^{\nu_1}, C_2 h^{\nu_2}) \le C h^\nu . \]
where $F(t)=F_1(t)F_2(t)$ is the unique decomposition of the sampled MVF. The claim follows since $d$ is majorized by $d_\psi$.
\end{proof}

There are several other matrix properties that we get directly from the definition of product operators. For example, the maximal eigenvalue of the matrices $\{ \mathcal{S}( \mathbf{A})(t) \mid  t \in \mathtt{I}\}$ is bounded by the maximal eigenvalues of the matrices in $\{ \mathcal{S}_1( \mathbf{A}_1)(t),\mathcal{S}_2( \mathbf{A}_2)(t) \mid  t \in \mathtt{I}\}$, and in general, 
\[   \| \mathcal{S}( \mathbf{A})(t) \| \le \| \mathcal{S}_1( \mathbf{A}_1)(t) \| \| \mathcal{S}_2( \mathbf{A}_2)(t) \| , \]
for any  sub-multiplicative matrix norm $\| \cdot \|$. 

We end this subsection by considering two additional properties: commuting with determinant, and homogeneity. We define them and show that having those on each component of approximation implies that the product operator also satisfies them. 

\begin{definition}[Commutativity with determinant] \label{def:CommutativityDet}
Let $\mathcal{S}$ be an operator defined on a class of matrices $\gmat = \gmat(n)$ of a given order $n$. We say that $\mathcal{S}$ \textbf{commutes with the determinant} if the following holds:
\begin{enumerate}
\item \label{eqn:determinant_well_defined} 
For any $M \in \gmat(n)$, $\det(M) \in \gmat(1)$.
\item \label{thm:determinant_commute} 
For each $t \in \mathtt{I}$,
\[  \mathcal{S} (\det(\mathbf{A}))(t) = \det( \mathcal{S} (\mathbf{A})(t))  . \]
\end{enumerate}
\end{definition}

To illustrate Definition \ref{def:CommutativityDet} consider the following two examples. First, any matrix in $\spd$ has a positive determinant. Such a determinant is positive definite matrix of order $1$. Thus, any scheme on $\spd$ can be tested for its commutativity with the determinant operator, see \cite{UriNir}. Second, any matrix in $\so$ is orthogonal and has determinant $1$, which is in $\son{1}$.  

Next we introduce a homogeneity property.
\begin{definition}[Homogeneity] \label{def:Homogeneity}
Let $\mathcal{S}$ be an operator defined on a class of matrices $\gmat=\gmat(n)$ of any order $n$, and let $\boldsymbol{\alpha} = \{ \alpha^{(i)} \}_{i \in \mathcal{J}}$ be such that $\boldsymbol{\alpha}  \subset \gmat(1)$ and $\boldsymbol{\alpha} \mathbf{A} = \{ \alpha^{(i)} A^{(i)} \} \subset \gmat(n)$. We say that $S$ is \textbf{homogeneous} if for any $t \in \mathtt{I}$
\[  \mathcal{S} (\boldsymbol{\alpha} \mathbf{A})(t) = \mathcal{S} (\boldsymbol{\alpha})(t) \mathcal{S} ( \mathbf{A})(t) . \]
\end{definition}

The above two properties are ``inherited" by the product operator, as is claimed next.
\begin{proposition} \label{prop:GeneralPropertiesProductScheme}
Let $\mathcal{S}$ be an approximation operator of the form \eqref{eqn:productSchemes}. 
\begin{enumerate}
\item
If $\mathcal{S}_1$ and $\mathcal{S}_2$ commute with the determinant, then, $S$ also commutes with the determinant.
\item
If $\mathcal{S}_1$ and $\mathcal{S}_2$ are homogeneous so is $\mathcal{S}$.
\end{enumerate}
\end{proposition}

\begin{proof}
In this proof we use the shorthand notation without the variable $t$. To prove $1$ , we note that since the determinant commutes with $\mathcal{S}_1$ and $\mathcal{S}_2$, we have
\begin{eqnarray*}
\det (\mathcal{S}( \mathbf{A})) &=& \det (\mathcal{S}_1( \mathbf{A}_1) \mathcal{S}_2(\mathbf{A}_2 )) =   \det (\mathcal{S}_1( \mathbf{A}_1)) \det( \mathcal{S}_2(\mathbf{A}_2 )) \\
  &=&   \mathcal{S}_1( \det( \mathbf{A}_1))  \mathcal{S}_2(\det(\mathbf{A}_2 )) = \mathcal{S}(\det(\mathbf{A})) .
\end{eqnarray*}
For the homogeneity, we assume $\boldsymbol{\alpha}  \subset \gmat(1)$. Thus, we can decompose any scalar to two scalars $\alpha^{(i)} =
\alpha^{(i)}_1 \alpha^{(i)}_2$ such that $\boldsymbol\alpha_j \subset \gmat_j(1)$, $j=1,2$. Therefore, claim $2$ follows since
\[
\mathcal{S}(\boldsymbol{\alpha} \mathbf{A})= \mathcal{S}_1(\{\alpha^{(i)}_1 A^{(i)}_1\}) \mathcal{S}_2(\{\alpha^{(i)}_2 A^{(i)}_2\}) .
\]
\end{proof}

%========================= Schemes for triangular matrices =======================

\section{``Square root'' operators} \label{sec:triangular_matrices}

We conclude the paper with a variant of the product operators which we term \textit{product data}. We present the method by an example which demonstrates how product data may reduces a relatively difficult approximation problem to easier one, based on matrix decomposition. In this example we examine the problem of approximation on the class of triangular matrices.

\subsection{Approximation of MVFs over triangular matrices}
The class of triangular matrices appears in both theory and applications, e.g., the transformation matrix between the standard basis to the Bernstein Basis \cite{itai2006eigenstructure}. Here, we consider square, lower triangular matrices of full rank, which forms a Lie group and is denoted by $\ltri$. For Lie groups data, approximation operators are available, see e.g., \cite{WallnerNiraProx}. Nonetheless, many of these operators might be difficult to apply in practice (details to follow). Therefore, we look for a relaxation of this problem using an appropriate matrix decomposition.

The Cholesky decomposition \cite[Chapter 4]{GolubMatrixBook} is defined for any positive semi-definite matrix and has the form $A = LL^\ast$ where $L$ is a lower triangular matrix. This decomposition is merely one variant of the LU decomposition, restricted to positive semi-definite matrices. The Cholesky decomposition is unique when $A \in \spd$, and then $L \in \ltri$.

Note that following the notion of product approximation operators (as presented so far), one might use the Cholesky decomposition to define operators on $\spd$ based upon known Lie groups operators on $\ltri$. However, it turns out that the construction of operators on $\ltri$ is more complicated than on $\spd$. Following are several arguments to support this statement. For geodesics based operators, such as in \cite{UriNir}, the Riemannian metric on $\spd$ provides a simple formula for the geodesics there, whereas the analogue in $\ltri$ is hard to compute, see e.g., \cite{marenich2002computational,marenich1997geodesics}. For operators based on the exp-log mapping, for example \cite{nava2011donoho}, the exponential and its inverse are relatively easy to evaluate for any $\spd $ matrix, while the analogues in $\ltri$ are challenging to compute, see e.g., \cite{moler1978nineteen}. Moreover, product schemes on $\ltri$ are not invariant under orthogonal coordinate change. At last, it is worth mentioning that a matrix in the class of triangular matrices might not be diagonalizable which result in many additional computational difficulties compared to diagonalizable matrices such as $\spd$ matrices. Therefore, constructing product operator for approximation on $\spd $ using the Cholesky decomposition is not a good idea.

We suggest a variant to our approach which also based both on matrix decomposition and the idea of using it to reduce the problem to easier domain. We do it by mapping the data from $\ltri$ to $\spd$. According to Cholesky decomposition this is a straightforward (and invertible) step, done by taking the product of any matrix with its transpose. Then, we apply an $\spd$ operator on the new data and evaluate our approximation by back-mapping the result based on the Cholesky decomposition. This procedure is summarized in Algorithm~\ref{alg:Cholesky}.

\begin{algorithm}[ht]
\caption{Product data approach for data on $\ltri$}
\label{alg:Cholesky}
\begin{algorithmic}[1]
\REQUIRE  The data $\{ L_i \}_{i \in \mathcal{J}} \subset \ltri$. \\ $\mathcal{S}$, an approximation operator on $\spd$. \\ An evaluation point, $t \in \mathtt{I} \subset \mathbb{R}$.
\ENSURE $\Gamma_L(t)$, where $\Gamma_L \colon \mathtt{I} \to \ltri$.
\FOR{$i \in \mathcal{J}$}    \label{alg:data_const}
\STATE  $A_{i} \gets L_iL_i^\ast$.  
\ENDFOR     \label{alg:data_const2}
\STATE   $\Gamma(t) \gets \mathcal{S}(\{ A_i \}_{i \in \mathcal{J}})(t)$. \\
\COMMENT{Denote by $\operatorname{chol}(\Gamma(t))$ the unique matrix $L \in \ltri$ such that $\Gamma(t)=LL^\ast$}  
\STATE $\Gamma_L(t) = \operatorname{chol}(\Gamma(t)) $.
\RETURN $\Gamma_L(t) $
\end{algorithmic}
\end{algorithm}

In view of Proposition~\ref{prop:GeneralPropertiesProductScheme}, we can further deduce the analogues extra properties in our example. First, if the $\spd$ operator commutes with the determinant, so is its induced $\ltri$ operator. Second, the homogeneity is inherited from the $\spd$ operator to the induced $\ltri$ operator. Moreover, the operator on $\ltri$ is invariant to orthogonal change of coordinates when the $\spd$ operator is and when the data is closed under such orthogonal coordinate changes. The formal proofs of these claims are in the spirit of the proof of Proposition~\ref{prop:GeneralPropertiesProductScheme}, and thus are omitted.

\subsection{Product data approach and its approximation order}

The general approach of product data assumes we have a (unique) matrix decomposition of the form \eqref{eqn:matrixDecomposition} where the given data is the sequence $\mathbf{A}_1  = \{ A^{(i)}_1 \}_{i \in \mathcal{J}} \subset \gmat_1$ and and mapping to and back from $\mathbf{A}  = \{ A^{(i)} \}_{i \in \mathcal{J}} \subset \gmat $ is done based on the matrix decomposition. As in the above example, the motivation is for cases where the available approximation operator on $\gmat$ is simpler than applying a direct operator on $\gmat_1$. 

To conclude this part we show that the approximation order can be ``translated'' from a known result on the operator of $\gmat$ to the approximation in a product data fashion on $\gmat_1$.

\begin{proposition}
Consider an approximation operator $\mathcal{S}$, defined on $(\gmat,d)$, with approximation order $\nu$. If $d_\psi$ is majorized by $d$ in $\gmat$, for any $\psi$, then the operator $\mathcal{S}_1$ based on the product data approach on $(\gmat_1,d_1)$ has approximation order $\nu$.
\end{proposition}
\begin{proof}
Denote by $F_1(t)$ the sampled MVF on $\gmat_1$. We use the fact that $d_1$ is always majorized by $d_\psi$ since 
\[ d_1(\mathcal{S}_1(\mathbf{A}_1)(t),F_1(t)) \le d_p(\mathcal{S}(\mathbf{A})(t),F(t)) \le c d_\psi(\mathcal{S}(\mathbf{A})(t),F(t)) ,  \]
where $d_p$ the $p$-metric \eqref{eqn:p_product_metric} on $\gmat$. The first inequality is by definition of $d_p$ and the second is based on metric equivalence with a constant $c$. This equivalent is deduced from the continuity of $\psi$, Lemma~\ref{lemma:equivalent}, and the compactness (approximation order is measured asymptotically for small neighbourhoods). Combining the latter with the majorization assumption we get for some constants $\tilde{c}$ and $C$, independent of $h$, that
\[ d_1\left(\mathcal{S}_1(\mathbf{A}_1)(t),F_1(t)\right) \le \tilde{c} d\left(\mathcal{S}(\mathbf{A})(t),F(t) \right) \le C h^\nu . \]
\end{proof}

\Acknowledgement The authors thank Roi Poranne for very helpful discussions about the geometrical interpretations of this paper.

\bibliography{BibPaper.bib}

\begin{thebibliography}{10}

\bibitem{Alexa2000}
Marc Alexa, Daniel Cohen-Or, and David Levin.
\newblock As-rigid-as-possible shape interpolation.
\newblock In {\em Proceedings of the 27th annual conference on Computer
  graphics and interactive techniques}, SIGGRAPH '00, pages 157--164, New York,
  NY, USA, 2000. ACM Press/Addison-Wesley Publishing Co.

\bibitem{ALMean}
Tsuyoshi Ando, Chi-Kwong Li, and Roy Mathias.
\newblock Geometric means.
\newblock {\em Linear Algebra Appl.}, 385:305--334, 2004.

\bibitem{arsigny2007geometric}
Vincent Arsigny, Pierre Fillard, Xavier Pennec, and Nicholas Ayache.
\newblock Geometric means in a novel vector space structure on symmetric
  positive-definite matrices.
\newblock {\em SIAM journal on matrix analysis and applications},
  29(1):328--347, 2007.

\bibitem{beals1984scattering}
Richard Beals and Ronald Coifman.
\newblock Scattering and inverse scattering for first order systems.
\newblock {\em Communications on Pure and Applied Mathematics}, 37(1):39--90,
  1984.

\bibitem{bunse1991numerical}
Angelika Bunse-Gerstner, Ralph Byers, Volker Mehrmann, and Nancy~K Nichols.
\newblock Numerical computation of an analytic singular value decomposition of
  a matrix valued function.
\newblock {\em Numerische Mathematik}, 60(1):1--39, 1991.

\bibitem{chefd2002constrained}
Christophe Chefd’hotel, David Tschumperl{\'e}, Rachid Deriche, and
  O~Faugeras.
\newblock Constrained flows of matrix-valued functions: Application to
  diffusion tensor regularization.
\newblock In {\em Computer Vision—ECCV 2002}, pages 251--265. Springer, 2002.

\bibitem{chen2003analysis}
Xin Chen, Houduo Qi, and Paul Tseng.
\newblock Analysis of nonsmooth symmetric-matrix-valued functions with
  applications to semidefinite complementarity problems.
\newblock {\em SIAM Journal on Optimization}, 13(4):960--985, 2003.

\bibitem{delsarte1979nevanlinna}
Philippe Delsarte, Yves Genin, and Yves~G. Kamp.
\newblock The {N}evanlinna-{P}ick problem for matrix-valued functions.
\newblock {\em SIAM Journal on Applied Mathematics}, 36(1):47--61, 1979.

\bibitem{dieci1999smooth}
Luca Dieci and Timo Eirola.
\newblock On smooth decompositions of matrices.
\newblock {\em SIAM Journal on Matrix Analysis and Applications},
  20(3):800--819, 1999.

\bibitem{global_nira}
Nira Dyn and Nir Sharon.
\newblock A global approach to the refinement of manifold data.
\newblock {\em Mathematics of Computation}, 2015.
\newblock To appear.

\bibitem{dyn2017manifold}
Nira Dyn and Nir Sharon.
\newblock Manifold-valued subdivision schemes based on geodesic inductive
  averaging.
\newblock {\em Journal of Computational and Applied Mathematics}, 311:54--67,
  2017.

\bibitem{elliott1970determination}
David Elliott.
\newblock Determination of finite strain and initial shape from deformed
  elliptical objects.
\newblock {\em Geological Society of America Bulletin}, 81(8):2221--2236, 1970.

\bibitem{feldmann1995reduced}
Peter Feldmann and Roland~W Freund.
\newblock Reduced-order modeling of large linear subcircuits via a block
  {L}anczos algorithm.
\newblock In {\em Proceedings of the 32nd annual ACM/IEEE Design Automation
  Conference}, pages 474--479. ACM, 1995.

\bibitem{GolubMatrixBook}
Gene~H. Golub and Charles~F. Van~Loan.
\newblock {\em Matrix computations}.
\newblock Johns Hopkins Studies in the Mathematical Sciences. Johns Hopkins
  University Press, Baltimore, MD, third edition, 1996.

\bibitem{itai2006eigenstructure}
Uri Itai.
\newblock On the eigenstructure of the bernstein kernel.
\newblock {\em Electron. Trans. Numer. Anal}, 25:431--438, 2006.

\bibitem{UriNir}
Uri Itai and Nir Sharon.
\newblock Subdivision schemes for positive definite matrices.
\newblock {\em Foundations of Computational Mathematics}, 13(3):347--369, 2013.

\bibitem{ay2010analytic}
Dasa Janovska and Vladimir Janovsky.
\newblock The analytic svd: on the non-generic points on the path.
\newblock {\em Electronic Transactions on Numerical Analysis}, 37:70--86, 2010.

\bibitem{marenich2002computational}
Andrey Marenich.
\newblock Computational geometry in {H}eisenberg group {H}eis\^{}3.
\newblock {\em ArXiv preprint math/0204105}, 2002.

\bibitem{marenich1997geodesics}
Valery Marenich.
\newblock Geodesics in {H}eisenberg groups.
\newblock {\em Geometriae Dedicata}, 66(2):175--185, 1997.

\bibitem{mathias1993approximation}
Roy Mathias.
\newblock Approximation of matrix-valued functions.
\newblock {\em SIAM journal on matrix analysis and applications},
  14:1061--1061, 1993.

\bibitem{moler1978nineteen}
Cleve Moler and Charles Van~Loan.
\newblock Nineteen dubious ways to compute the exponential of a matrix.
\newblock {\em SIAM review}, 20(4):801--836, 1978.

\bibitem{nava2011donoho}
Esfandiar Nava-Yazdani and Thomas~PY Yu.
\newblock On {D}onoho's {L}og-{E}xp subdivision scheme: choice of retraction
  and time-symmetry.
\newblock {\em Multiscale Modeling \& Simulation}, 9(4):1801--1828, 2011.

\bibitem{schonemann1968two}
Peter~H. Sch{\"o}nemann.
\newblock On two-sided orthogonal procrustes problems.
\newblock {\em Psychometrika}, 33(1):19--33, 1968.

\bibitem{NirUri}
Nir Sharon and Uri Itai.
\newblock Approximation schemes for functions of positive-definite matrix
  values.
\newblock {\em IMA Journal of Numerical Analysis}, 33(4):1436--1468, 2013.

\bibitem{shingel2009interpolation}
Tania Shingel.
\newblock Interpolation in special orthogonal groups.
\newblock {\em IMA Journal of Numerical Analysis}, 29(3):731--745, 2009.

\bibitem{sinap1994polynomial}
Ann Sinap and Walter Van~Assche.
\newblock Polynomial interpolation and {G}aussian quadrature for matrix-valued
  functions.
\newblock {\em Linear algebra and its applications}, 207:71--114, 1994.

\bibitem{sorkine2007rigid}
Olga Sorkine and Marc Alexa.
\newblock As-rigid-as-possible surface modeling.
\newblock In {\em Symposium on Geometry processing}, pages 109--116. Citeseer,
  2007.

\bibitem{WallnerNiraProx}
Johannes Wallner and Nira Dyn.
\newblock Convergence and {$C^1$} analysis of subdivision schemes on manifolds
  by proximity.
\newblock {\em Comput. Aided Geom. Design}, 22(7):593--622, 2005.

\bibitem{ThomasLie2}
Gang Xie and Thomas P.-Y. Yu.
\newblock Smoothness equivalence properties of manifold-valued data subdivision
  schemes based on the projection approach.
\newblock {\em SIAM J. Numer. Anal.}, 45(3):1200--1225 (electronic), 2007.

\end{thebibliography}
\bibliographystyle{plain}

\end{document}